\documentclass[11pt]{amsart}

\usepackage{amsmath,amssymb,amsthm,amscd,mathrsfs}

\usepackage{geometry}
\usepackage{graphicx}
\usepackage{subfig}
\usepackage{tikz}
\usepackage{extarrows} 

\usepackage{enumitem}
\usepackage{url,verbatim}
\usepackage[normalem]{ulem}
\usepackage{fixme}
\usepackage{centernot}
\usepackage{xcolor}
\RequirePackage[colorlinks,citecolor=blue,urlcolor=blue]{hyperref}



\calclayout

\usepackage[colorlinks,citecolor=blue,urlcolor=blue]{hyperref}

\theoremstyle{plain}
\newtheorem{theorem}{Theorem}[section]
\newtheorem{definition}[theorem]{Definition}
\newtheorem{lemma}[theorem]{Lemma}
\newtheorem{proposition}[theorem]{Proposition}
\newtheorem{corollary}[theorem]{Corollary}
\newtheorem{example}[theorem]{Example}

\newtheorem{remark}[theorem]{Remark}

\numberwithin{equation}{section}
\numberwithin{figure}{section}

\newcommand\PP{{\mathbb P}}

\newcommand\eps{\epsilon}

\renewcommand\ell{l}

\title[Recursive supercritical bounds for site percolation]{%
Recursive Packing Bounds for Supercritical Disconnection in Bernoulli Site Percolation%
}

\author{Zhongyang Li}
\address{Department of Mathematics,
University of Connecticut,
Storrs, Connecticut 06269-3009, USA}
\email{zhongyang.li@uconn.edu}
\urladdr{\url{https://mathzhongyangli.wordpress.com}}

\begin{document}
\maketitle

\begin{abstract}
For Bernoulli site percolation on an infinite, connected, locally finite graph
$G=(V,E)$, we obtain quantitative upper bounds on the supercritical disconnection
probability
\[
\mathbb{P}_p(S\nleftrightarrow\infty)
\]
for arbitrary finite or infinite sets $S\subset V$ and all $p>p^{\mathrm{site}}_c(G)$.

The key quantity is a recursive packing number $\mathbf{PK}_{p,\eps,c}(S)$. It is the
maximal number of vertices that can be extracted from $S$ so that, after deleting witness
balls around the previously chosen vertices, each selected vertex still connects to infinity
with probability at least $c$, while its failure to connect to infinity is already detected,
up to a factor $1+\eps$, by failure to reach the inner boundary of its witness ball. Thus
$\mathbf{PK}_{p,\eps,c}(S)$ counts essentially independent local witnesses for the global
event $\{S\nleftrightarrow\infty\}$.

We prove the structural estimate
\[
\mathbb{P}_p(S\nleftrightarrow\infty)
\le
\frac{\eps(1-c)}{c}
+(1-c)^{\mathbf{PK}_{p,\eps,c}(S)}.
\]
Combining this bound with the local functional characterization of $p^{\mathrm{site}}_c(G)$
from \cite{ZL24} yields an explicit supercritical estimate valid on every infinite,
connected, locally finite graph.

We also illustrate the packing number on ray-homogeneous trees. In particular, sparse
finite subsets of a distinguished ray have packing number equal to their cardinality, both
for regular trees and for a non-regular decorated spine. This shows that the packing
number is explicit on concrete graph families.
\end{abstract}

\section{Introduction}

\subsection{Model and question}

Let $G=(V,E)$ be an infinite, connected, locally finite graph. In Bernoulli site
percolation with parameter $p\in(0,1)$, each vertex is declared open with probability $p$
independently of all others, and closed otherwise; we write $\mathbb{P}_p$ for the
resulting product measure. For a vertex $v\in V$, we write
\[
\{v\leftrightarrow\infty\}
\]
for the event that $v$ belongs to an infinite open cluster.

The critical site percolation threshold is
\[
p^{\mathrm{site}}_c(G)
:=
\sup\Bigl\{
p\in[0,1]:
\forall v\in V,\ \mathbb{P}_p(v\leftrightarrow\infty)=0
\Bigr\}.
\]
Fix a finite or infinite set $S\subset V$. We write
\[
\{S\nleftrightarrow\infty\}
:=
\bigcap_{v\in S}\{v\nleftrightarrow\infty\}.
\]
The purpose of this paper is to obtain quantitative upper bounds on
\[
\mathbb{P}_p(S\nleftrightarrow\infty)
\]
throughout the supercritical regime $p>p^{\mathrm{site}}_c(G)$.

\subsection{A packing bound for supercritical disconnection}

Our main tool is a recursive packing number. Given $p,\eps,c\in(0,1)$ and $S\subset V$,
the quantity $\mathbf{PK}_{p,\eps,c}(S)$ is the maximal number of vertices that can be
selected recursively from $S$ so that, after removing witness balls around the previously
chosen vertices, each selected vertex still connects to infinity with probability at least
$c$, and the event that it does not connect to infinity is already captured, up to a factor
$1+\eps$, by the event that it does not reach the inner boundary of its witness ball; see
Definition~\ref{df33} below.

In this language, the heart of the paper is the following packing bound.

\begin{theorem}[Packing bound for disconnection]
\label{thm:main-intro}
Let $G=(V,E)$ be an infinite, connected, locally finite graph. Let $p,\eps,c\in(0,1)$ and
let $S\subset V$. Then
\begin{equation}
\label{eq:intro-structural}
\mathbb{P}_p(S\nleftrightarrow\infty)
\le
\frac{\eps(1-c)}{c}
+(1-c)^{\mathbf{PK}_{p,\eps,c}(S)}.
\end{equation}
\end{theorem}

Thus each additional admissible packing center contributes one more factor $(1-c)$.

The explicit supercritical bound proved later is obtained by choosing $c$ through the local
functional characterization of $p^{\mathrm{site}}_c(G)$. More precisely, if
$p>p^{\mathrm{site}}_c(G)$, $p_1\in(p^{\mathrm{site}}_c(G),p)$, and $\eps,\delta\in(0,1)$,
then one may take
\begin{align}
c
=
1-\Bigl(\frac{1-p}{1-p_1}\Bigr)^{1-\eps},\label{chc}
\end{align}
which yields
\begin{equation}
\label{eq:intro-explicit}
\mathbb{P}_p(S\nleftrightarrow\infty)
\le
\frac{\delta\left(\frac{1-p}{1-p_1}\right)^{1-\eps}}
{1-\left(\frac{1-p}{1-p_1}\right)^{1-\eps}}
+\left(\frac{1-p}{1-p_1}\right)^{
(1-\eps)\mathbf{PK}_{p,\delta,\,
1-\left(\frac{1-p}{1-p_1}\right)^{1-\eps}}(S)
}.
\end{equation}

The novelty of the present paper is the recursive packing/amplification
principle that turns the one-point supercritical information in \cite{ZL231} into a many-point
disconnection estimate.

\subsection{What does the packing number count?}

The packing number is not an arbitrary geometric packing parameter. For
\[
H_c:=\{v\in V:\mathbb{P}_p(v\leftrightarrow\infty)\ge c\},
\]
Proposition~\ref{p24} shows that
\[
\mathbf{1}_{\{S\cap H_c\neq\varnothing\}}
\le
\mathbf{PK}_{p,\eps,c}(S)
\le
|S\cap H_c|.
\]
Hence $\mathbf{PK}_{p,\eps,c}(S)$ measures how many uniformly supercritical vertices in
$S$ can be organized into an essentially independent family of local witnesses for the
global event $\{S\nleftrightarrow\infty\}$.
Moreover, when $c$ is chosen by (\ref{chc}), it was proved by Lemma 3.3 in \cite{ZL231} that $|H_c|=\infty$. 

This interpretation is also the reason for the recursive nature of the definition. The witness
balls are chosen one after another, and at each step the relevant connectivity event is
measured in the graph punctured by the previously used balls. This allows us to exploit
independence between the local boundary-disconnection events while retaining a uniform
positive probability of connection to infinity at every stage.

\subsection{Method and relation to previous work}

Quantitative supercritical bounds of the form
\[
\mathbb{P}_p(S\nleftrightarrow\infty)\le e^{-f(S)}
\]
have been obtained in several settings under additional geometric assumptions; see in
particular \cite{DC20,EST25}. Our result is complementary: it applies to every infinite,
connected, locally finite graph, and packages the geometry into the graph-dependent
packing number $\mathbf{PK}_{p,\eps,c}(S)$.

The proof combines two ingredients. The first is the local functional
$\varphi_p^v(S)$ from \cite{ZL24}, which characterizes $p^{\mathrm{site}}_c(G)$ through
finite-volume connectivity data and furnishes vertices with a quantitative lower bound on
their supercritical connection probability. The second is an amplification step: by
selecting suitable packing centers and comparing disconnection from infinity with
disconnection from a local inner boundary, one turns a uniform lower bound on one-point
connection probabilities into the exponential estimate \eqref{eq:intro-structural}.

A related functional/cutset amplification mechanism also appears in \cite{ZL26} in the
context of planar site percolation and uniqueness.

\subsection{Examples and organization}

To show that the packing quantity is concrete, and not merely formal, we study it on
ray-homogeneous trees. In that setting one can compute the packing number exactly on
explicit sparse subsets of a distinguished ray: if the points are sufficiently separated, then
the packing number equals the cardinality of the set. This applies both to regular trees and
to a non-regular decorated spine. In particular, the method is not restricted to transitive or
quasi-transitive graphs.

Section~\ref{sec2} recalls the functional $\varphi_p^v(S)$ from \cite{ZL24}, proves the
packing bound \eqref{eq:intro-structural}, and derives the explicit supercritical estimate
\eqref{eq:intro-explicit}. Section~\ref{sec3} is devoted to examples on ray-homogeneous trees,
where the packing number can be computed exactly on suitable finite sets.

\section{Functional characterization of $p_c^{\mathrm{site}}$ and packing-based supercritical
disconnection bounds}\label{sec2}

This section recalls the local functional $\varphi_p^v(S)$ from \cite{ZL24}, introduces the
recursive packing number, and proves the packing bound of Theorem~\ref{thm:main-intro}.
We then derive the explicit supercritical estimate stated in Corollary~\ref{cor:explicit-main}.

\subsection{Definition of the $\varphi_p^v(S)$-functional}

Let $G=(V,E)$ be a graph. For each $p\in (0,1)$, let $\mathbb{P}_p$ be the probability measure of the i.i.d.~Bernoulli($p$) site percolation on $G$.
For each $S\subset V$, let $S^{\circ}$ consist of all the interior vertices of $S$, i.e., vertices all of whose neighbors are in $S$ as well.
For each  $S\subseteq V$, $v\in S$, define
\begin{align}
\varphi_p^{v}(S):=\begin{cases}\sum_{y\in S:[\partial_V y]\cap S^c\neq\emptyset}\mathbb{P}_p(v\xleftrightarrow{S^{\circ}} \partial_V y)&\mathrm{if}\ v\in S^{\circ}\\
1&\mathrm{if}\ v\in S\setminus S^{\circ}
\end{cases}\label{dpv}
\end{align}
where 
\begin{itemize}
\item $v\xleftrightarrow{S^{\circ}} x$ is the event that the vertex $v$ is joined to the vertex $x$ by an open path visiting only interior vertices in $S$;
\item let $A\subseteq V$; $v\xleftrightarrow{S^{\circ}} A$ if and only if there exists $x\in A$ such that $v\xleftrightarrow{S^{\circ}} x$;
\item $\partial_V y$ consists of all the vertices adjacent to $y$.
\end{itemize}

\subsection{Characterization theorem from \cite{ZL24}}

The following technical lemmas were proved in \cite{ZL24}; see also \cite{ZL231}.

\begin{lemma}(Lemma 2.5 in \cite{ZL24})\label{l71} Let $G=(V,E)$ be an infinite, connected, locally finite graph. The critical site percolation probability on $G$ is given by
\begin{align*}
    \tilde{p}_c=\sup\{p\geq 0:\exists \epsilon_0>0, \mathrm{s.t.}\forall v\in V, \exists S_v\subseteq V\ \mathrm{satisfying}\ |S_v|<\infty\ \mathrm{and}\ v\in S_v^{\circ}, \varphi_p^{v}(S_v)\leq 1-\epsilon_0\}
\end{align*}
Moreover,
\begin{enumerate}
    \item If $p>\tilde{p}_c$, a.s.~there exists an infinite 1-cluster; moreover, for any $\epsilon>0$ there exists a vertex $w$,  such that 
    \begin{align}
    \forall S_w\subseteq V\ \mathrm{satisfying}\ |S_w|<\infty\ \mathrm{and}\ w\in S_{w}^{\circ},\ \varphi_q^v(S_w)> 1-\epsilon_1\label{wc1};\ \forall q\geq p_1
    \end{align}
    where $p_1,\epsilon_1$ are such that 
    \begin{align}
    p_1\in (\tilde{p}_c,p);\ \epsilon_1\in(0,\epsilon);\ \left(\frac{1-p}{1-p_1}\right)^{1-\epsilon_1}<\left(\frac{1-p}{1-\tilde{p}_c}\right)^{1-\epsilon}.\label{pepe}
    \end{align}
    Any vertex $w$ satisfying (\ref{wc1}) also satisfies
    \begin{align}
    \mathbb{P}_p(w\leftrightarrow \infty)\geq 1-\left(\frac{1-p}{1-p_1}\right)^{1-\epsilon_1}\geq 1-\left(\frac{1-p}{1-\tilde{p}_c}\right)^{1-\epsilon}.\label{lbc}
    \end{align}
    \item If $p<\tilde{p}_c$, then for any vertex $v\in V$
    \begin{align}
        \mathbb{P}_p(v\leftrightarrow \infty)=0.\label{lbc2}
    \end{align}
\end{enumerate}
In particular, (1) and (2) imply that $p_c^{site}(G)=\tilde{p}_c$.
\end{lemma}

\subsection{Packing definition and quantitative upper bounds}

\begin{definition}Let $G=(V,E)$ be a graph. Let $D\geq 1$ be a positive integer. For each $v\in V$, let
\begin{align*}
B(v,D):=\{u\in V:d_{G}(v,u)\leq D\}.
\end{align*}
and
\begin{align*}
\partial_{in}B(v,D):=\{u\in B(v,D):\exists w\notin B(v,D),\ s.t.~(v,w)\in E\}.
\end{align*}
\end{definition}

\begin{definition}\label{df33}Let $G=(V,E)$ be an infinite, connected, locally finite graph. Let $S\subset V$ be a (finite or infinite) set of vertices. For $p,\epsilon,c\in(0,1)$, we define the $(p,\epsilon,c)$-packing number $\mathbf{PK}_{p,\epsilon,c}(S)$ to be the maximal cardinality of subsets 
\begin{align*}
V_S:=\{w_1,w_2,\ldots\}\subset S
\end{align*}
satisfying all the following conditions: 
\begin{enumerate}
\item for every $w_i\in V_S$, there exists $D_i>0$, such that 
\begin{align*}
w_i\in V\setminus[\cup_{j=1}^{i-1}B(w_j,D_j)]
\end{align*}
\item Let $G_i:=G\setminus [\cup_{j=1}^{i}B(w_j,D_j)]$; with $G_0:=G$, such that
\begin{align}
\PP_p(w_i\nleftrightarrow \partial_{in}B(w_i,D_i);G_{i-1})\le\PP_p(w_i\nleftrightarrow \infty;G_{i-1})\leq \PP_p(w_i\nleftrightarrow \partial_{in}B(w_i,D_i);G_{i-1})(1+\epsilon)\label{ctd}
\end{align}
and
\begin{align}
\PP_p(w_i\leftrightarrow\infty;G_{i-1})\geq c>0\label{wil};
\end{align}
\end{enumerate}
where $\{w_i\leftrightarrow\infty;G_{i-1}\}$ is the event that $w_i$ is in an infinite-1 cluster in the Bernoulli site percolation of the graph $G_{i-1}$.
\end{definition}

In words, $\mathbf{PK}_{p,\varepsilon,c}(S)$ is the largest number of vertices that can be chosen
recursively from $S$ so that, after removing the witness balls already used, each selected
vertex still has a uniform probability at least $c$ of connecting to infinity, and the global
event that this vertex does not connect to infinity is already captured, up to the factor
$1+\varepsilon$, by the local event that it does not reach the inner boundary of its witness
ball. Thus the packing number counts how many essentially disjoint local witnesses for
the global event $\{S \nleftrightarrow \infty\}$ can be extracted from $S$.

\begin{proposition}\label{p24} (Immediate bounds and non-vacuity of the packing number).
Let
\[
H_c:=\{w\in V: P_p(w\leftrightarrow\infty)\ge c\}.
\]
Then for every $S\subset V$,
\[
\mathbf 1_{\{S\cap H_c\neq\varnothing\}}
\le \mathbf{PK}_{p,\varepsilon,c}(S)
\le |S\cap H_c|.
\]
\end{proposition}

\begin{proof}
For the upper bound, let $\{w_1,w_2,\dots\}$ be an admissible family in Definition \ref{df33}.
Then for every $i$,
\[
P_p(w_i\leftrightarrow\infty;G)\ge P_p(w_i\leftrightarrow\infty;G_{i-1})\ge c,
\]
hence $w_i\in H_c$. Therefore the whole packing family is contained in $S\cap H_c$, and
\[
\mathbf{PK}_{p,\varepsilon,c}(S)\le |S\cap H_c|.
\]

For the lower bound, assume $w\in S\cap H_c$. Since
\[
\{w\nleftrightarrow \partial_{\mathrm{in}} B(w,D)\}\uparrow \{w\nleftrightarrow \infty\}
\quad\text{as }D\to\infty,
\]
there exists $D$ large enough such that
\[
P_p(w\nleftrightarrow \partial_{\mathrm{in}} B(w,D))
\le P_p(w\nleftrightarrow \infty)
\le (1+\varepsilon)P_p(w\nleftrightarrow \partial_{\mathrm{in}} B(w,D)).
\]
Since also $P_p(w\leftrightarrow\infty)\ge c$, the singleton $\{w\}$ is admissible in
Definition \ref{df33}. Hence $\mathbf{PK}_{p,\varepsilon,c}(S)\ge 1$.
\end{proof}

\begin{remark}[Role of Lemma \ref{l71} in the packing construction]
Lemma \ref{l71} identifies a concrete source of vertices with a uniform lower bound on the
supercritical connection probability. More precisely, if $p>p_c^{\rm site}(G)$,
$p_1\in (p_c^{\rm site}(G),p)$, and $w$ satisfies the uniform functional condition in
(\ref{wc1}), then (\ref{lbc}) yields
\[
P_p(w\leftrightarrow\infty)\ge 1-\Bigl(\frac{1-p}{1-p_1}\Bigr)^{1-\varepsilon_1}.
\]
Hence Lemma \ref{l71} furnishes the vertices that may serve as centers in the packing
construction.

Furthermore, in settings where the argument of Lemma 3.3 in \cite{ZL231} applies, one
obtains infinitely many vertices satisfying the corresponding strong local functional
condition; in particular, there are infinitely many vertices enjoying the above global lower
bound. Thus, in those settings, the packing construction can draw from an infinite
reservoir of candidate centers.
\end{remark}

\begin{lemma}\label{lm34}Let $G=(V,E)$ be an infinite, connected, locally finite graph. Let $S\subset V$. Then
\begin{align*}
\mathbb{P}_p(S\nleftrightarrow \infty)\leq \frac{(1-c)\epsilon}{c}+(1-c)^{\mathbf{PK}_{p,\epsilon,c}(S)}
\end{align*}
\end{lemma}

\begin{proof}

When $\mathbf{PK}_{p,\epsilon,c}(S)=1$,
the conclusion of the lemma follows from (\ref{wil}).

Let $k\geq 2$. For $1\leq i\leq k$, define
\begin{align*}
A_i:=\{w_i\nleftrightarrow\infty;G_{i-1}\};\qquad
B_i:=\cap_{j=1}^{i}A_j.
\end{align*}
and
\begin{align*}
A_{i,D_i}:=\{w_i\nleftrightarrow \partial_{in}B(w_i,D_i);G_{i-1}\};\qquad B_{i,\mathbf{D}_i}=\cap_{j=1}^{i}A_{j,D_j}.
\end{align*}
where
\begin{align*}
\mathbf{D}_i=\{D_1,\ldots,D_i\}
\end{align*}

We have
\begin{align}
\mathbb{P}_p(B_k)&=\PP_p(A_k\cap B_{k-1})\label{ab1}\\
&=\PP_p(A_{k,{D_k}}\cap B_{k-1,\mathbf{D}_{k-1}})+\PP_p(A_{k}\cap A_{k,D_k}^c\cap B_{k-1,\mathbf{D}_{k-1}})+\PP_p(A_{k}\cap B_{k-1}\cap B_{k-1,\mathbf{D}_{k-1}}^c)\notag
\end{align}
Note that $\{A_{i,D_i}\}_{i=1}^k$ are mutually independent, therefore
\begin{align}
\PP_p(A_{k,D_k}\cap B_{k-1,\mathbf{D}_{k-1}})= \prod_{i=1}^k\PP_p(A_{i,D_i})\leq(1-c)^k\label{ab2}
\end{align}
where the last inequality follows from (\ref{ctd}), (\ref{wil}).
Moreover by independence we have
\begin{align}
\PP_p(A_{k}\cap A_{k,{D_k}}^c\cap B_{k-1,\mathbf{D}_{k-1}})=
\PP_p(B_{k-1,\mathbf{D}_{k-1}})\PP_p(A_{k}\cap A_{k,D_k}^c)\label{ab3}
\end{align}
By (\ref{ctd}), we have
\begin{align}
\PP_p(A_{k}\cap A_{k,D_k}^c)\leq \epsilon \PP_p(A_{k,D_k})\label{ab4}
\end{align}

It follows from (\ref{ab3})-(\ref{ab4}) that
\begin{align}
\PP_p(A_{k}\cap A_{k,D_k}^c\cap B_{k-1,\mathbf{D}_{k-1}})\leq \epsilon \PP_p(B_{k-1,\mathbf{D}_{k-1}})\PP_p(A_{k,D_k})=\epsilon \prod_{j=1}^k\PP_p(A_{j,D_j})\le\epsilon(1-c)^k\label{ab6}
\end{align}
We also have 
\begin{align}
&\PP_p(A_{k}\cap B_{k-1}\cap B_{k-1,\mathbf{D}_{k-1}}^c)\label{ab7}\le
\PP_p(B_{k-1}\cap B_{k-1,\mathbf{D}_{k-1}}^c)
\end{align}
By (\ref{ab1}), (\ref{ab2}), (\ref{ab6}), (\ref{ab7}) we have
\begin{align}
\PP_p(B_k)\leq (1+\epsilon)\PP_p(B_{k,\mathbf{D}_{k}})+\PP_p(B_{k-1}\cap B_{k-1,\mathbf{D}_{k-1}}^c),\label{abA}
\end{align}
so
\begin{align}
\PP_p(B_k)-\PP_p(B_{k,\mathbf{D}_k})\leq \epsilon \PP_p(B_{k,\mathbf{D}_{k}})+\PP_p(B_{k-1}\cap B_{k-1,\mathbf{D}_{k-1}}^c)\label{ab8}
\end{align}
For $i\geq 1$, if we write
\begin{align*}
Q_i:=\PP_p(B_i)-\PP_p(B_{i,\mathbf{D}_i})=\PP_p(B_{i}\cap B_{i,\mathbf{D}_{i}}^c)
\end{align*}
Then (\ref{ab8}) and (\ref{ab2}) gives the following recursive inequality
\begin{align}
\label{abb}
Q_k\leq \epsilon (1-c)^k+ Q_{k-1}
\end{align}
By (\ref{ctd}) we have
\begin{align}
\label{abc}
Q_1\leq \epsilon (1-c)
\end{align}

By (\ref{abb}) and (\ref{abc}) we infer that
\begin{align}
\PP_p(B_k)-\PP_p(B_{k,\mathbf{D}_k})\leq \epsilon \left[\sum_{i=1}^k (1-c)^i\right]= \frac{\epsilon[(1-c)-(1-c)^{k+1}]}{c}\label{ab9}
\end{align}
for all $k\geq 1$.

Then by (\ref{abA}) we have
\begin{align}
\PP_p(B_k)\leq (1+\epsilon)(1-c)^k+\frac{\epsilon(1-c)}{c}-\frac{\epsilon(1-c)^k}{c}\label{pdk}
=(1-c)^k+\frac{\epsilon(1-c)}{c}\left[1-(1-c)^k\right]
\end{align}
Note that
\begin{align*}
\PP_p(S\nleftrightarrow\infty)\leq \PP_p(B_k),
\end{align*}
then the lemma follows.
\end{proof}

\begin{corollary}[Explicit supercritical disconnection bound]
\label{cor:explicit-main}
Let $G=(V,E)$ be an infinite, connected, locally finite graph, let
$p>p^{\mathrm{site}}_c(G)$, and let $S\subset V$ be any finite or infinite set of vertices.
Then
\[
\mathbb{P}_p(S\nleftrightarrow\infty)
\le
\inf_{\substack{p_1\in(p^{\mathrm{site}}_c(G),\,p)\\ \eps,\delta\in(0,1)}}
\left[
\frac{\delta\left(\frac{1-p}{1-p_1}\right)^{1-\eps}}
{1-\left(\frac{1-p}{1-p_1}\right)^{1-\eps}}
+\left(\frac{1-p}{1-p_1}\right)^{
(1-\eps)\mathbf{PK}_{p,\delta,\,
1-\left(\frac{1-p}{1-p_1}\right)^{1-\eps}}(S)}
\right].
\]
\end{corollary}

\begin{proof}
Apply Lemma~\ref{lm34} with
\[
c:=1-\left(\frac{1-p}{1-p_1}\right)^{1-\eps}.
\]
Lemma~\ref{l71} guarantees that such choices yield a nontrivial lower bound on one-point
connection probabilities in the supercritical regime.
\end{proof}

\section{Examples on ray-homogeneous trees}\label{sec3}

This section illustrates the packing number on a concrete class of trees. We first show
that on any ray-homogeneous tree, sparse finite subsets of a distinguished ray have
packing number equal to their cardinality. We then record two concrete instances:
regular trees and a non-regular decorated spine. These examples show that the packing
construction is explicit and not restricted to transitive graphs.

\begin{definition}[Ray-homogeneous tree along a ray]
\label{def:ray-homogeneous}
Let $T$ be an infinite, locally finite tree, and let
\[
\gamma=(y_0,y_1,y_2,\dots)
\]
be a fixed geodesic ray in $T$. For each $n\ge 1$, let $T_n^+$ denote the connected
component of $y_n$ in the graph obtained from $T$ by deleting the edge
$\langle y_{n-1},y_n\rangle$.

We say that $T$ is \emph{ray-homogeneous along $\gamma$} if there exist rooted trees
$(\mathbb F_0,o_0)$ and $(\mathbb F_1,o_1)$ such that
\begin{enumerate}
\item for every $n\ge 0$,
\[
(T,y_n)\cong (\mathbb F_0,o_0),
\]
\item for every $n\ge 1$,
\[
(T_n^+,y_n)\cong (\mathbb F_1,o_1).
\]
\end{enumerate}
\end{definition}

\begin{remark}
\label{rem:ray-homogeneous}
The point of Definition~\ref{def:ray-homogeneous} is that the first selected vertex in the
packing construction sees the rooted model $\mathbb F_0$, while every later selected vertex,
after the preceding witness ball cuts the ray just before it, sees the rooted model
$\mathbb F_1$. This is exactly the geometric input needed in the proof below.
\end{remark}

Suppose now that $T$ is ray-homogeneous along $\gamma$ in the sense of
Definition~\ref{def:ray-homogeneous}. Let $p\in(0,1)$. For $j\in\{0,1\}$ and $R\ge 1$, write
\[
a_j(R):=\PP_p^{\mathbb F_j}\bigl(o_j\nleftrightarrow \partial_{\mathrm{in}}B_{\mathbb F_j}(o_j,R)\bigr),
\qquad
a_j:=\PP_p^{\mathbb F_j}(o_j\nleftrightarrow \infty),
\]
and
\[
\theta_j:=1-a_j=\PP_p^{\mathbb F_j}(o_j\leftrightarrow \infty).
\]
Since
\[
\{o_j\nleftrightarrow \partial_{\mathrm{in}}B_{\mathbb F_j}(o_j,R)\}\uparrow
\{o_j\nleftrightarrow \infty\}
\qquad\text{as }R\to\infty,
\]
we have
\[
a_j(R)\uparrow a_j
\qquad\text{as }R\to\infty.
\]

\begin{proposition}[Exact packing on sparse subsets of a ray]
\label{prop:ray-homogeneous-packing}
Let $T$ be ray-homogeneous along a ray
\[
\gamma=(y_0,y_1,y_2,\dots),
\]
with associated rooted trees $(\mathbb F_0,o_0)$ and $(\mathbb F_1,o_1)$.
Assume that
\[
p>\max\bigl\{p_c^{\rm site}(\mathbb F_0),\, p_c^{\rm site}(\mathbb F_1)\bigr\},
\]
and let $\varepsilon\in(0,1)$.

Choose an integer $R_\varepsilon\ge 1$ such that
\begin{equation}
\label{eq:R-eps-choice-0}
a_0\le (1+\varepsilon)a_0(R_\varepsilon)
\end{equation}
and
\begin{equation}
\label{eq:R-eps-choice-1}
a_1\le (1+\varepsilon)a_1(R_\varepsilon).
\end{equation}

Let
\[
A=\{y_{n_1},y_{n_2},\dots,y_{n_k}\}\subset \gamma
\]
with
\[
0\le n_1<n_2<\cdots<n_k
\]
and assume that
\begin{equation}
\label{eq:ray-gap}
n_{i+1}-n_i\ge R_\varepsilon+1
\qquad\text{for all }1\le i\le k-1.
\end{equation}
Then for every
\[
c\in \bigl(0,\min\{\theta_0,\theta_1\}\bigr],
\]
one has
\[
\mathbf{PK}_{p,\varepsilon,c}(A)=|A|=k.
\]
More generally, if $S\subset V(T)$ contains $A$, then
\[
\mathbf{PK}_{p,\varepsilon,c}(S)\ge |A|.
\]
\end{proposition}

\begin{proof}
Let
\[
w_i:=y_{n_i},
\qquad 1\le i\le k.
\]
For $1\le i\le k-1$, define
\[
D_i:=n_{i+1}-n_i-1,
\]
and define
\[
D_k:=R_\varepsilon.
\]
By \eqref{eq:ray-gap}, we have
\[
D_i\ge R_\varepsilon
\qquad\text{for all }1\le i\le k.
\]

We claim that the ordered family
\[
(w_1,D_1),\dots,(w_k,D_k)
\]
is admissible in the sense of Definition~\ref{df33}.

\medskip

\noindent
\textit{Step 1: verification of condition \textnormal{(1)} in Definition~\ref{df33}.}

Fix $i\in\{1,\dots,k\}$. If $j<i$, then
\[
d_T(w_j,w_i)=n_i-n_j.
\]
Since $n_i\ge n_{j+1}$ for every $j<i$, we get
\[
n_i-n_j\ge n_{j+1}-n_j > n_{j+1}-n_j-1 = D_j.
\]
Hence
\[
w_i\notin B(w_j,D_j)
\qquad\text{for every }j<i.
\]
Therefore
\[
w_i\in V(T)\setminus \bigcup_{j=1}^{i-1} B(w_j,D_j),
\]
which is exactly condition \textnormal{(1)}.

\medskip

\noindent
\textit{Step 2: the first vertex $w_1$.}

Since $(T,w_1)\cong (\mathbb F_0,o_0)$ by Definition~\ref{def:ray-homogeneous}, we have
\[
\PP_p(w_1\nleftrightarrow \infty;T)=a_0
\]
and
\[
\PP_p\bigl(w_1\nleftrightarrow \partial_{\mathrm{in}}B(w_1,D_1);T\bigr)=a_0(D_1).
\]
Because $D_1\ge R_\varepsilon$ and $a_0(R)$ is increasing in $R$,
\[
a_0(R_\varepsilon)\le a_0(D_1)\le a_0.
\]
Using \eqref{eq:R-eps-choice-0}, we obtain
\[
\PP_p\bigl(w_1\nleftrightarrow \partial_{\mathrm{in}}B(w_1,D_1);T\bigr)
\le
\PP_p(w_1\nleftrightarrow \infty;T)
\le
(1+\varepsilon)\,
\PP_p\bigl(w_1\nleftrightarrow \partial_{\mathrm{in}}B(w_1,D_1);T\bigr).
\]
Moreover,
\[
\PP_p(w_1\leftrightarrow \infty;T)=\theta_0\ge c.
\]
Thus conditions (\ref{ctd}) and (\ref{wil}) hold for $i=1$.

\medskip

\noindent
\textit{Step 3: the later vertices $w_i$, $2\le i\le k$.}

Fix $i\in\{2,\dots,k\}$. Let
\[
G_{i-1}:=T\setminus \bigcup_{j=1}^{i-1} B(w_j,D_j).
\]
We claim that the connected component of $w_i$ in $G_{i-1}$ is exactly $T_{n_i}^+$.

Indeed, since
\[
D_{i-1}=n_i-n_{i-1}-1,
\]
the ball $B(w_{i-1},D_{i-1})$ contains the ray vertices
\[
y_{n_{i-1}},y_{n_{i-1}+1},\dots,y_{n_i-1},
\]
but does not contain $y_{n_i}=w_i$. Therefore the predecessor $y_{n_i-1}$ of $w_i$ along
the ray $\gamma$ has been removed, while $w_i$ itself has not. On the other hand, if
$j\le i-2$, then every vertex of $T_{n_i}^+$ is at distance at least $n_i-n_j>D_j$ from
$w_j$, so the earlier balls $B(w_j,D_j)$ do not meet $T_{n_i}^+$. Hence the connected
component of $w_i$ in $G_{i-1}$ is precisely $T_{n_i}^+$.

By Definition~\ref{def:ray-homogeneous}, this component is rooted-isomorphic to
$(\mathbb F_1,o_1)$. Therefore
\[
\PP_p(w_i\nleftrightarrow \infty;G_{i-1})=a_1
\]
and
\[
\PP_p\bigl(w_i\nleftrightarrow \partial_{\mathrm{in}}B(w_i,D_i);G_{i-1}\bigr)=a_1(D_i).
\]
Since $D_i\ge R_\varepsilon$,
\[
a_1(R_\varepsilon)\le a_1(D_i)\le a_1.
\]
Using \eqref{eq:R-eps-choice-1}, we get
\[
\PP_p\bigl(w_i\nleftrightarrow \partial_{\mathrm{in}}B(w_i,D_i);G_{i-1}\bigr)
\le
\PP_p(w_i\nleftrightarrow \infty;G_{i-1})
\le
(1+\varepsilon)\,
\PP_p\bigl(w_i\nleftrightarrow \partial_{\mathrm{in}}B(w_i,D_i);G_{i-1}\bigr),
\]
and also
\[
\PP_p(w_i\leftrightarrow \infty;G_{i-1})=\theta_1\ge c.
\]
Thus conditions (\ref{ctd}) and (\ref{wil}) hold for every $2\le i\le k$.

We have proved that the family $\{w_1,\dots,w_k\}$ is admissible. Hence
\[
\mathbf{PK}_{p,\varepsilon,c}(A)\ge k.
\]

For the reverse inequality, $\mathbf{PK}_{p,\varepsilon,c}(A)\le |A|$, is straightforward from Definition \ref{df33}.

Therefore
\[
\mathbf{PK}_{p,\varepsilon,c}(A)=|A|=k.
\]

Finally, if $S\subset V(T)$ contains $A$, the same admissible family yields
\[
\mathbf{PK}_{p,\varepsilon,c}(S)\ge |A|.
\]
This completes the proof.
\end{proof}

\begin{corollary}[Regular trees]
\label{cor:regular-tree-special-case}
Let $T_d$ be the infinite $d$-regular tree, where $d\ge 3$, and let
\[
\gamma=(y_0,y_1,y_2,\dots)
\]
be a fixed geodesic ray. Then $T_d$ is ray-homogeneous along $\gamma$ in the sense of
Definition~\ref{def:ray-homogeneous}.

More precisely, one may take
\[
(\mathbb F_0,o_0)=(T_d,y_0),
\]
and $(\mathbb F_1,o_1)$ to be the rooted $(d-1)$-ary tree. For $p>1/(d-1)$, define
\[
u_0:=1-p,\qquad
u_{n+1}:=1-p+p\,u_n^{\,d-1},
\qquad
u:=\lim_{n\to\infty}u_n.
\]
Then
\[
a_0(R)=1-p+p\,u_{R-1}^{\,d},
\qquad
a_0=1-p+p\,u^d,
\]
and
\[
a_1(R)=1-p+p\,u_{R-1}^{\,d-1},
\qquad
a_1=1-p+p\,u^{d-1}.
\]
Equivalently,
\[
\theta_0=p(1-u^d),
\qquad
\theta_1=p(1-u^{d-1}).
\]

Hence, for every $\eps\in(0,1)$, there exists $R_\eps\ge 1$ such that for every finite set
\[
A=\{y_{n_1},\dots,y_{n_k}\}\subset \gamma
\]
satisfying
\[
n_{i+1}-n_i\ge R_\eps+1
\qquad\text{for all }1\le i\le k-1,
\]
one has
\[
\mathbf{PK}_{p,\eps,c}(A)=|A|
\qquad\text{for every }c\in(0,\theta_1].
\]
\end{corollary}

\begin{proof}
The ray-homogeneous statement is immediate from the symmetry of the regular tree.
The formulas for $a_0(R),a_0,a_1(R),a_1$ are the standard branching-process recursions
for site percolation on $T_d$. The conclusion follows from
Proposition~\ref{prop:ray-homogeneous-packing}, since $\theta_1\le \theta_0$.
\end{proof}

\begin{example}[A non-regular decorated spine]
\label{ex:decorated-spine}
Let $(\mathbb T_3,o)$ be the infinite rooted ternary tree, i.e. every vertex has exactly
three children. Let $L=(x_n)_{n\in\mathbb Z}$ be a bi-infinite path, and construct a tree
$T$ by attaching to each $x_n$ one copy of $\mathbb T_3$ via an edge joining $x_n$ to the
root of that copy. Let
\[
\gamma=(x_0,x_1,x_2,\dots)
\]
be the positive spine ray.

Then $T$ is ray-homogeneous along $\gamma$, but $T$ is not regular: the spine vertices
have degree $3$, whereas non-root vertices in the attached ternary trees have degree $4$.
If $(\mathbb F_0,o_0)$ and $(\mathbb F_1,o_1)$ denote the rooted models from
Definition~\ref{def:ray-homogeneous}, then
\[
p_c^{\rm site}(T)
=
p_c^{\rm site}(\mathbb F_0)
=
p_c^{\rm site}(\mathbb F_1)
=
\frac13.
\]

Consequently, for every $p>\frac13$ and every $\eps\in(0,1)$, there exists $R_\eps\ge 1$
such that every finite set
\[
A=\{x_{n_1},x_{n_2},\dots,x_{n_k}\}\subset \gamma
\]
satisfying
\[
n_{i+1}-n_i\ge R_\eps+1
\qquad\text{for all }1\le i\le k-1
\]
obeys
\[
\mathbf{PK}_{p,\eps,c}(A)=|A|
\qquad\text{for every }c\in\bigl(0,\min\{\theta_0,\theta_1\}\bigr].
\]
\end{example}

\begin{proof}
The ray-homogeneous property is immediate from translation along the spine.

To compute the critical threshold, first note that if $p>\frac13$, then the root of the copy
of $\mathbb T_3$ attached to $x_0$ has positive probability to connect to infinity inside that
copy. Hence $x_0$ connects to infinity in $T$ with positive probability, so
\[
p_c^{\rm site}(T)\le \frac13.
\]
The same argument applies to $\mathbb F_1$, since its root also carries one attached copy
of $\mathbb T_3$.

Conversely, if $p<\frac13$, then site percolation on $\mathbb T_3$ is subcritical, so every
open cluster inside a fixed attached copy of $\mathbb T_3$ is almost surely finite. Moreover,
the intersection of the open cluster of $x_0$ with the spine is dominated by the open cluster
of the origin for Bernoulli site percolation on $\mathbb Z$, which is almost surely finite for
every $p<1$. Hence only finitely many attached copies can be reached from $x_0$, and the
whole open cluster of $x_0$ in $T$ is almost surely finite. Therefore
\[
p_c^{\rm site}(T)\ge \frac13.
\]
Exactly the same argument on the one-sided spine gives
\[
p_c^{\rm site}(\mathbb F_1)\ge \frac13.
\]
Since $(\mathbb F_0,o_0)\cong (T,x_0)$, we also have
\[
p_c^{\rm site}(\mathbb F_0)=p_c^{\rm site}(T).
\]
Thus
\[
p_c^{\rm site}(T)
=
p_c^{\rm site}(\mathbb F_0)
=
p_c^{\rm site}(\mathbb F_1)
=
\frac13.
\]

The final packing statement now follows directly from
Proposition~\ref{prop:ray-homogeneous-packing}.
\end{proof}

\bibliographystyle{plain}
\bibliography{rf,psg}
\end{document}